\documentclass[11pt]{article}

\title{Homotopy Types of Intervals in Corank-Three Higher Bruhat Orders}
\author{Daria Poliakova}
\date{}

\usepackage[a4paper,margin=2.8cm]{geometry}
\usepackage{amsmath,amssymb,amsthm}
\usepackage{mathtools}
\usepackage{graphicx}
\usepackage{enumitem}
\usepackage[hidelinks]{hyperref}

\usepackage{tikz}
\usetikzlibrary{arrows.meta,calc,positioning}

\setlist[itemize]{leftmargin=2em}
\setlist[enumerate]{leftmargin=2.2em}

\usepackage{accents}
\newlength{\dhatheight}

\newtheorem{theorem}{Theorem}[section]
\newtheorem{proposition}[theorem]{Proposition}
\newtheorem{lemma}[theorem]{Lemma}
\newtheorem{corollary}[theorem]{Corollary}
\newtheorem{conjecture}[theorem]{Conjecture}

\theoremstyle{definition}
\newtheorem{definition}[theorem]{Definition}

\begin{document}

\maketitle

\begin{abstract}
We prove Reiner's conjecture for higher Bruhat orders in corank \(3\): the facial intervals of \(B(n,n-3)\) are precisely the spherical intervals, and all other intervals are contractible.
\end{abstract}

\section*{Introduction}

Higher Bruhat orders \(B(n,d)\) were introduced by Manin and Schechtman
as higher-dimensional generalizations of the classical weak order on the
symmetric group, which is the case \(B(n,1)\)
\cite{ManinSchechtman1989}. The weak order exhibits particularly nice
topological behavior: its intervals are spherical if and only if they
correspond to faces of the permutahedron, whereas all other intervals are
contractible \cite{Bjorner1984}.

Unlike the weak order, higher Bruhat orders are not polytopal in general:
Felsner and Ziegler showed that already the cover graph of \(B(6,3)\) is
not the \(1\)-skeleton of any convex polytope
\cite[Observation~5.3]{FelsnerZiegler2001}. Nevertheless, Ziegler's
oriented-matroid description \cite{Ziegler1993} still allows one to define
so-called facial intervals, which for \(d=1\) are precisely the intervals
corresponding to faces of the permutahedron. Higher Bruhat orders can also be
described as posets of fine zonotopal tilings of cyclic zonotopes
\cite{RichterGebertZiegler1994,FelsnerZiegler2001}, and in this language,
facial intervals consist of all fine tilings refining a fixed zonotopal
subdivision.

Reiner then conjectured that these are precisely the spherical intervals
for arbitrary \(d\), whereas all other intervals are contractible
\cite{Reiner1999}. Sphericity of facial intervals is known in full
generality, while contractibility of nonfacial intervals remained
conjectural. McConville proved the full conjecture for \(d=2\)
\cite{McConville2018}.

There are two things that make general \(d\) more difficult than the
classical \(d=1\) case. First, the weak order is a lattice, whereas general
\(B(n,d)\) need not be: already \(B(6,2)\) is not a lattice
\cite{Ziegler1993}. Second, the weak order on permutations is precisely the
inclusion order on inversion sets, whereas for general \(B(n,d)\) this is
no longer true. Ziegler exhibited two elements of \(B(8,3)\) whose
inversion sets are contained in one another, but which are not comparable
in the higher Bruhat order \cite{Ziegler1993}. Felsner and Weil proved that
this phenomenon does not occur for \(d=2\): \(B(n,2)\) is ordered by
inclusion for every \(n\) \cite{FelsnerWeil2000}. This is one reason why
the case \(d=2\) is more tractable.

However, in addition to low dimensions, there is one more class of cases
where the inclusion order coincides with the higher Bruhat order. Ziegler
proved that this happens in low corank, namely when \(n-d\leq4\)
\cite[Theorem~4.5]{Ziegler1993}. Of these cases, \(n-d\leq3\), which was studied recently by Chau from an enumerative viewpoint \cite{Chau2024}, is
particularly nice, because there the lattice property also holds
\cite[Theorem~4.4]{Ziegler1993}. For poset topology, this makes the
Crosscut Theorem available. One can thus argue that the \(B(n,n-3)\) case is easier than the \(B(n,2)\) case solved by McConville.

In this note, we make use of the observations above and prove Reiner's
conjecture in corank \(3\).

\subsubsection*{Acknowledgements} I am grateful to Vadim Schechtman and Gleb Koshevoy for encouraging me to think about higher Bruhat orders, and to Eleni Tzanaki for encouraging me to think about poset topology.

\subsubsection*{AI use declaration} ChatGPT Pro (models 5.5 and 5.6 Sol) was used as search engine, proof assistant, vector graphics artist and editor. All the outputs of the proof assistant were thoroughly digested by human. Exposition in the current note is by human and for humans.

\section{Statement of Reiner's conjecture}

\subsection{Arbitrary dimension}
Higher Bruhat orders were introduced by Manin and Schechtman
\cite{ManinSchechtman1989}. We use Ziegler's equivalent description in
terms of inversion sets \cite{Ziegler1993}.

\begin{definition}\label{def:higher-bruhat}
Let \(K=\{k_1<\cdots<k_{d+2}\}\subset [n]\). Its packet is the ordered list
\[
P(K)=
\bigl(
K\setminus\{k_{d+2}\},
K\setminus\{k_{d+1}\},
\ldots,
K\setminus\{k_1\}
\bigr)
\]
of \((d+1)\)-subsets of \(K\). A set
\(I\subseteq \binom{[n]}{d+1}\) is called consistent if, for every
\(K\in\binom{[n]}{d+2}\), the intersection \(I\cap P(K)\) is a beginning
or ending segment of the packet \(P(K)\).

The elements of \(B(n,d)\) are the consistent subsets of
\(\binom{[n]}{d+1}\). The order is generated by single-step inclusions
\(I\lessdot I\cup\{A\}\), where both sets are consistent. We call \(I\)
the inversion set.
\end{definition}

There is another equivalent description in terms of extensions of cyclic
hyperplane arrangements. 

\begin{definition}\label{def:cyclic-arrangement}
Let \(m\ge 1\). Choose real numbers \(t_1<\cdots<t_n\). The cyclic
hyperplane arrangement \(X_c^{n,m}\) is the arrangement of hyperplanes
\(H_1,\ldots,H_n\) in \(\mathbb R^m\) given by
\[
H_i=
\{(x_1,\ldots,x_m):
x_1+t_i x_2+\cdots+t_i^{m-1}x_m+t_i^m=0\}.
\]
Thus the normal vector of \(H_i\) is \( (1,t_i,t_i^2,\ldots,t_i^{m-1}) \). 

The combinatorial type does not depend on the choice of the parameters
\(t_i\), only on their order.
\end{definition}

A vertex of \(X_c^{n,m}\) is the intersection of \(m\) hyperplanes. We
label the vertex \(H_{j_1}\cap\cdots\cap H_{j_m}\) by the complementary
set \([n]\setminus\{j_1,\ldots,j_m\}\).

\begin{proposition}[Ziegler]\label{prop:ziegler-extension}
Let \(r=n-d\). The higher Bruhat order \(B(n,d)\) of Definition~\ref{def:higher-bruhat} is
isomorphic to the poset of uniform extensions of the cyclic arrangement \(X_c^{n,r-1}\) from Definition~\ref{def:cyclic-arrangement} by a
new oriented pseudohyperplane, ordered by single-step inclusion of the
vertices on the negative side.

Under this isomorphism, the inversion set of an extension \(h\) is the
set of labels of vertices lying on the negative side of \(h\).
\end{proposition}

This is \cite[Theorem~4.1]{Ziegler1993}.

This description singles out a special class of intervals.

\begin{definition}\label{def:facial-interval}
An interval \([I,J]\) in \(B(n,d)\) is called facial if it is the set of all
uniform extensions refining some fixed non-uniform extension.
\end{definition}

This is the extension-space definition of facial intervals from \cite[Section~6]{Reiner1999}.

For an interval \([I,J]\), let \(\Delta(I,J)\) denote the simplicial complex
whose simplices are the chains in the open interval \((I,J)\). We call
\([I,J]\) spherical if \(\Delta(I,J)\) is homotopy equivalent to a sphere,
and contractible if \(\Delta(I,J)\) is contractible. We regard the empty
complex as \(S^{-1}\).

\begin{conjecture}[Reiner]\label{conj:reiner}
Let \(I < J\). If \([I,J]\) is facial, then it is spherical. If \([I,J]\) is not facial,
then it is contractible.
\end{conjecture}

This is \cite[Conjecture~6.9]{Reiner1999}.

The first assertion is known: facial intervals are products of smaller higher
Bruhat orders and have the expected spherical homotopy type~\cite[after Conjecture~6.9]{Reiner1999}. Thus only the contractibility assertion in Conjecture~\ref{conj:reiner} remains to be proved.

\subsection{The corank 3 case}

We now specialize to \(d=n-3\). The cyclic arrangement \(X_c^{n,2}\) is an
arrangement of \(n\) lines in \(\mathbb R^2\). A uniform extension is obtained
by adding an oriented pseudoline \(h\) which meets every old line once and
avoids all old intersection points.

We label the intersection of lines \(i\) and \(j\) by \(ij\). It is convenient
to use this complementary notation instead of the corresponding
\((n-2)\)-subset \([n]\setminus\{i,j\}\). The extension \(h\) is then recorded
by the graph
\[
G_h=\{ij:ij\text{ lies on the negative side of }h\}.
\]

\begin{definition}\label{def:coconsistent}
A graph \(G\subseteq\binom{[n]}2\) is coconsistent if, for every \(i\in[n]\),
the neighbors of \(i\) form a beginning or ending segment of
\(1<\cdots<i-1<i+1<\cdots<n\).
\end{definition}

By Proposition~\ref{prop:ziegler-extension}, the uniform extensions of \(X_c^{n,2}\), and hence the elements of \(B(n,n-3)\), are precisely the coconsistent graphs of Definition~\ref{def:coconsistent}.

A non-uniform extension is of one of three types:
\begin{enumerate}
    \item the new pseudoline passes through some old intersection points, no two
    lying on the same old line;
    \item the new pseudoline is a repeated copy of one of the old lines;
    \item the new element is a loop.
\end{enumerate}

The first two types of non-uniform extensions together with intervals of their uniform refinements are shown in Figure~\ref{fig:nonuniform-extensions}, for \(n=5\). The uniform refinements of the loop extension are {\em all} the uniform extensions; the picture of full \(B(5,2)\) is Figure 3 in \cite{Ziegler1993}.

 Reiner's conjecture then asserts that the three types of non-uniform extensions above give rise to spherical intervals, whereas all other intervals are contractible. 

\begin{figure}[t]
    \centering

    \includegraphics[width=.74\textwidth]
    {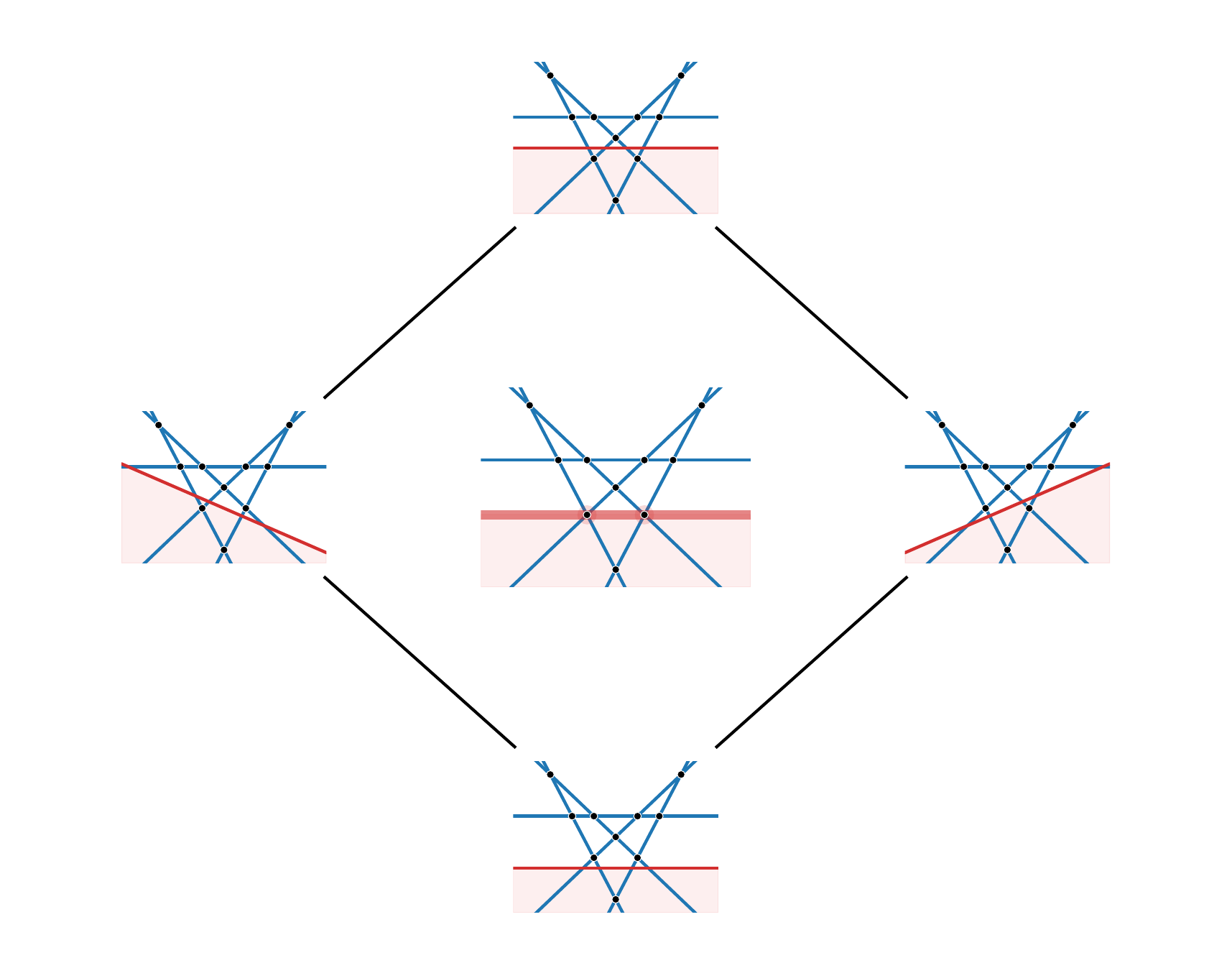}

    \vspace{1em}

    \includegraphics[width=.68\textwidth]
    {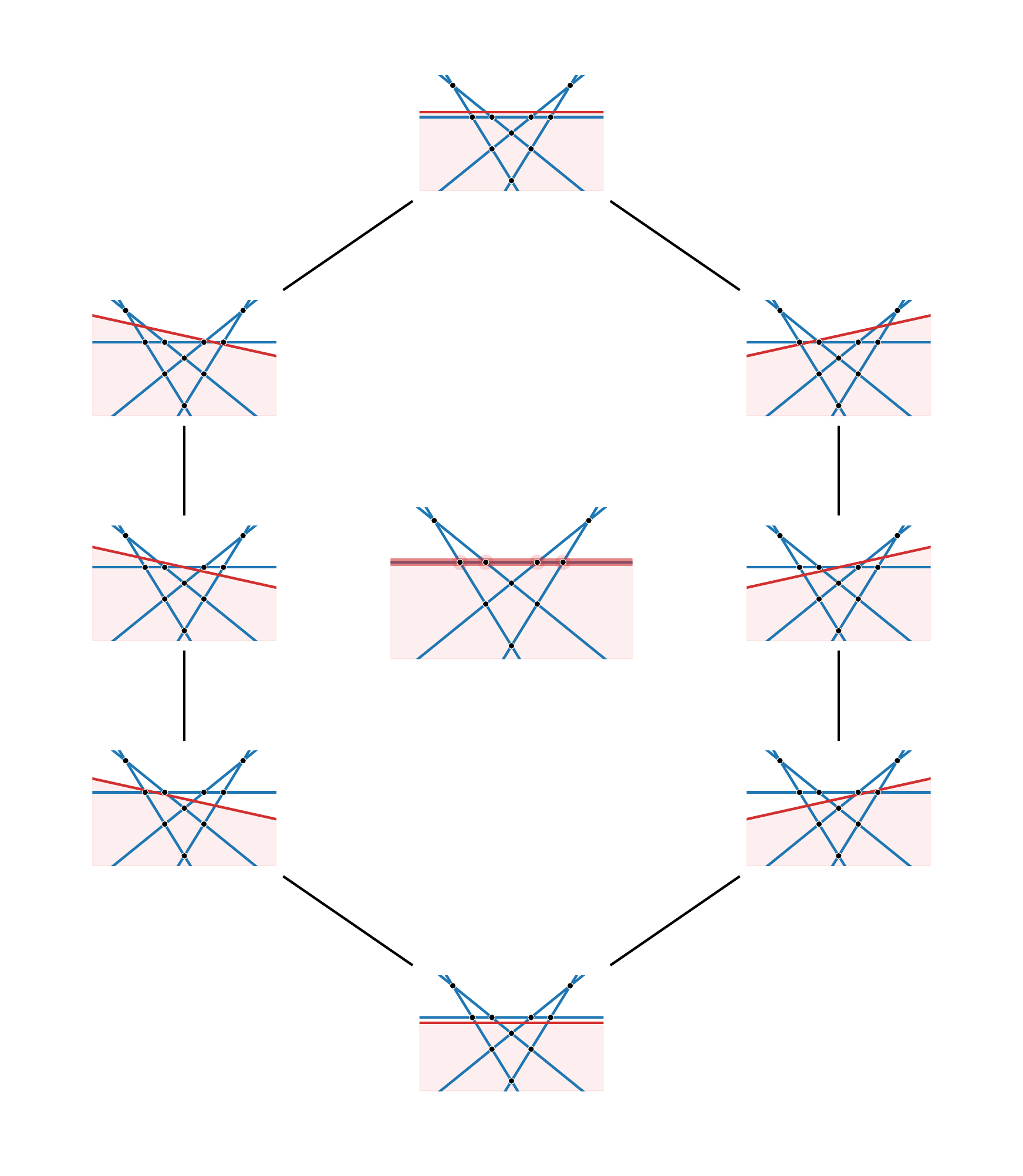}

    \caption{Two facial intervals in \(B(5,2) \), corresponding to uniform refinements of two non-uniform extensions.
    In each picture the non-uniform extension is shown in the center.
    Top: the new pseudoline passes through two old vertices.
    Bottom: the new pseudoline is a repeated copy of an old line.}
    \label{fig:nonuniform-extensions}
\end{figure}

\section{Four useful results}

In this section, we recall the four external results that we use in our proof.

\begin{proposition}[Localization theorem]\label{prop:localization}
A labeling of the vertices of a line arrangement by \(+\), \(0\), and \(-\)
comes from a single-element extension if and only if, on every old line, it
has one of the forms
\[
-\cdots-+\cdots+,\qquad
-\cdots-0+\cdots+,
\]
or the reverse of one of these, or consists entirely of zeros.
\end{proposition}

This is the rank-three specialization of Las Vergnas'
characterization of single-element extensions \cite{LasVergnas1978};
see also \cite[Theorem~7.1.8]{BLSWZ1999} for the modern formulation.

\begin{proposition}[Poset properties of \(B(n,n-3)\)]\label{prop:lattice}
The order on \(B(n,n-3)\) is inclusion of inversion sets, and
\(B(n,n-3)\) is a lattice.
\end{proposition}

This is \cite[Theorems~4.4 and~4.5]{Ziegler1993}.

\begin{proposition}[Crosscut theorem]\label{prop:crosscut}
Let \([I,J]\) be an interval in a finite lattice, and let \(A\) be its set of
atoms, i.e. elements covering \(I\). Define the crosscut simplicial complex
\[
\Gamma(I,J)=\{S\subseteq A:\bigvee S<J\}.
\]
Then \(\Delta(I,J)\) and \(\Gamma(I,J)\) are homotopy equivalent.
\end{proposition}

See \cite[Theorem~10.8]{Bjorner1995}.

\begin{proposition}[Ascent-block theorem]\label{prop:ascent-block}
For \(X\in B(n,d)\), set \(\operatorname{Asc}(X)=\{I\in\binom{[n]}{d+1}\setminus X:X\cup\{I\}\in B(n,d)\}\). Two ascents are called adjacent if their intersection has size \(d\). There is a decomposition of \(\operatorname{Asc}(X)\) into blocks of adjacent ascents:
\[
\operatorname{Asc}(X)=A_1\sqcup\cdots\sqcup A_N
\]
such that each \(A_t\) has the form
\[
A_t=
\bigl\{
\{a^t_1,\ldots,a^t_{d+1}\},
\{a^t_2,\ldots,a^t_{d+2}\},
\ldots,
\{a^t_{r_t},\ldots,a^t_{d+r_t}\}
\bigr\},
\]
where \(a^t_1<\cdots<a^t_{d+r_t}\). If \(I\in A_s\) and \(J\in A_t\) with \(s\neq t\), then \(|I\cap J|<d\).
\end{proposition}

This is \cite[Lemma~3.1]{McConville2018}.

\section{Proof}
\subsection{Facial intervals in terms of difference graphs}

We first characterize facial intervals by their difference graphs.

For \(p\in[n]\), let \(E_p=\{pq:q\neq p\}\) be the full star at \(p\).

\begin{proposition}\label{prop:facial-difference}
Let \(G < H\). An interval \([G,H]\) is facial if and only if \(H\setminus G\) is one of
the following:
\begin{enumerate}
    \item a nonempty matching;
    \item a full star \(E_p\);
    \item the complete graph \(K_n\).
\end{enumerate}
\end{proposition}

\begin{proof}
Suppose first that \([G,H]\) is the set of uniform extensions refining a
non-uniform extension. The edges of \(H\setminus G\) are the old vertices on
the new pseudoline, or, in the loop case, all old vertices. On every old line there are either no such vertices, one
such vertex, or all of them. Thus every vertex of \(H\setminus G\) has degree
\(0\), \(1\), or \(n-1\).

If no vertex has degree \(n-1\), the graph is a matching. If exactly one
vertex has degree \(n-1\), the graph is its full star. If two vertices have
degree \(n-1\), then every vertex has degree \(n-1\), so the graph is \(K_n\).

Conversely, suppose that \(H\setminus G\) is a matching, a full star, or
\(K_n\). Label an old vertex negative if it belongs to \(G\), zero if it
belongs to \(H\setminus G\), and positive otherwise. Since \(G\) and \(H\)
are coconsistent, on every old line the negative vertices and the
negative-or-zero vertices are beginning or ending segments. Moreover, the
number of zero vertices on each old line is \(0\), \(1\), or all. Proposition~\ref{prop:localization} therefore gives a non-uniform extension with this
labeling.

Its uniform refinements are precisely the coconsistent graphs \(F\) with
\(G\subseteq F\subseteq H\). Hence \([G,H]\) is facial.
\end{proof}

\subsection{Possible ascent-blocks and atom configurations}

We now apply McConville's ascent-block lemma in corank \(3\), using the complementary graph
language for ascents. In this language, McConville's adjacency is the usual
adjacency of edges: two ascents are adjacent if and only if the corresponding
edges share a vertex.

\begin{proposition}\label{prop:block-types}
Every ascent-block is of one of the following forms:
\begin{enumerate}
    \item a \(1\)-block \(\{ij\}\);
    \item a \(2\)-block
    \[
    \{12,1n\},\qquad
    \{1p,pn\}\quad (1<p<n),\qquad
    \{1n,(n-1)n\};
    \]
    \item the unique \(3\)-block
    \[
    \{12,1n,(n-1)n\}.
    \]
\end{enumerate}
Different blocks are vertex-disjoint.
\end{proposition}

\begin{proof}
By Proposition~\ref{prop:ascent-block}, the \((n-2)\)-subsets in a block are consecutive
subsets of their ordered support. A block with \(r\) elements has support of
size \(n-3+r\), so \(r\leq 3\).

There is only one possible \(3\)-block. Its support is \([n]\), so the block consists
of \( \{1,\ldots,n-2\} \),
\(\{2,\ldots,n-1\}\),
\(\{3,\ldots,n\}\). These correspond to the edges \((n-1)n\), \(1n\), and \(12\).

A \(2\)-block is determined by an \((n-1)\)-element support
\([n]\setminus\{p\}\). If \(p=n\), the two consecutive subsets are
\( \{1,\ldots,n-2\}, \{2,\ldots,n-1\} \),
and the corresponding edges are \((n-1)n\) and \(1n\). If \(p=1\), the
corresponding edges are \(12\) and \(1n\). If \(1<p<n\), they are \(1p\)
and \(pn\). 

Finally, a \(1\)-block is just one edge.

Proposition~\ref{prop:ascent-block} also says that ascents in different blocks intersect in
fewer than \(n-3\) elements. Two edges sharing a vertex
correspond to \((n-2)\)-subsets intersecting in exactly \(n-3\) elements.
Therefore edges in different blocks are vertex-disjoint.
\end{proof}

We now classify the possible configurations of atoms in one interval. Let \(A(G,H)=\{e\in H\setminus G:G\cup\{e\}\text{ is coconsistent}\}\); these edges index the atoms of \([G,H]\), and \(A(G,H)\subseteq\operatorname{Asc}(G)\).

\begin{proposition}\label{prop:atom-types}
The set \(A(G,H)\) is one of the following:
\begin{enumerate}
    \item a matching;
    \item a \(2\)-block together with a possibly empty disjoint matching;
    \item the \(3\)-block \(\{12,1n,(n-1)n\}\).
\end{enumerate}
\end{proposition}

\begin{proof}
By Proposition~\ref{prop:block-types}, if at most one edge is taken from each block, then the atoms form a matching,
since different blocks are vertex-disjoint.

Every \(2\)-block uses both vertices \(1\) and \(n\), so two \(2\)-blocks cannot
coexist. A \(2\)-block can coexist only with disjoint \(1\)-blocks, which form a
matching. Two edges chosen from the \(3\)-block either form one of the
\(2\)-blocks or are disjoint.

Finally, suppose that the whole \(3\)-block is contained in
\(\operatorname{Asc}(G)\). Addability of \(12\) and \(1n\) means no edges of the type \(1i\) are in \(G\), and addability of \(1n\) and \((n-1)n\) means no edges of the type \(in\) are in \(G\). An edge \(ij\), with \(1<i<j<n\), would make the set of neighbors of \(i\) in \(G\) contain \(j\) but neither \(1\) nor \(n\), so it would be neither a
beginning nor an ending segment. Hence \(G=\varnothing\). Since
\(\operatorname{Asc}(\varnothing)=\{12,1n,(n-1)n\}\), the \(3\)-block cannot
coexist with any other ascent.
\end{proof}

\subsection{Crosscut complex computations}
By Propositions~\ref{prop:lattice} and~\ref{prop:crosscut}, the homotopy type of \(\Delta(G,H)\) is the homotopy type of the crosscut complex \(\Gamma(G,H)\). We compute it in the three cases of Proposition~\ref{prop:atom-types}.

\subsubsection{Case 1: A matching}

\begin{lemma}\label{lem:matching-join}
Suppose that \(A(G,H)\) is a matching. Then, for every
\(S\subseteq A(G,H)\),
\[
\bigvee_{e\in S}(G\cup\{e\})=G\cup S.
\]
\end{lemma}

\begin{proof}
Since \(S\) is a matching, at most one edge of \(S\) is incident with each
vertex. Thus, at every vertex, the neighborhood in \(G\cup S\) is either the same as in \(G\) or the same as in \(G\cup\{e\}\) for some \( e \in S \). Hence \(G\cup S\) is coconsistent, and therefore it is the join.
\end{proof}

\begin{corollary}\label{cor:matching-case}
If \(H=G\cup A(G,H)\), then the crosscut complex \(\Gamma(G,H)\) is the boundary
of a simplex, and \([G,H]\) is a facial interval. Otherwise the crosscut complex
is a full simplex, and \([G,H]\) is contractible.
\end{corollary}

\begin{proof}
In the first case, every proper subset of the atoms has join below \(H\), while
the join of all atoms is \(H\). In this case \(H\setminus G\) is a matching, so the
interval is facial by Proposition~\ref{prop:facial-difference}.

In the second case, even the join of all atoms is below \(H\). Hence every subset
of the atoms is a face of the crosscut complex.
\end{proof}

\subsubsection{Case 2: A 2-block plus a (possibly empty) matching of 1-blocks}

Suppose now that we are in the second case, where \(A(G,H)\) consists of a \(2\)-block \(P\)
and a (possibly empty) disjoint matching \(M\). We first prove a lemma about when \(P\) is a genuine 2-block of \(G\), not a part of the 3-block.

\begin{proposition}[Forcing lemma]
\label{prop:forcing}
Let \(P \) be a \(2\)-block of \(\operatorname{Asc}(G)\), and let \(J\) be the join of two atoms of \(P\). Then every ascent of \(G\) outside \(P\) lies below \(J\).
\end{proposition}
\begin{proof}
First let \(P=\{12,1n\}\). Addability of both edges implies that \(1\) is
isolated in \(G\), while coconsistency forces \(J\) to contain the full star
of \(1\). Let \(ab\) be any ascent of \(G\) outside \(P\). It belongs to a different block and is therefore disjoint from \(P\), so \(2<a<b<n\). Since \(ab\) is addable and \(1a\notin G\), coconsistency
implies that \(an\in G\). Now \(J\) contains both \(1a\) and \(an\), so it
also contains \(ab\).

Next let \(P=\{1p,pn\}\), where \(1<p<n\). Again, \(p\) is isolated in
\(G\), and \(J\) contains the full star of \(p\). Let \(ab\) be any ascent of \(G\) outside \(P\). It belongs to a different block and is therefore disjoint from \(P\).

First assume \(a<b<p\). The graph \(G\) contains neither \(ab\) nor \(ap\),
and \(ab\) is addable, so coconsistency implies that \(1a\in G\). Now \(J\)
contains both \(1a\) and \(ap\), so it contains \(ab\).

Next assume \(a<p<b\). Again \(G\) contains neither \(ab\) nor \(ap\), and
\(ab\) is addable, so coconsistency implies that \(an\in G\). Now \(J\)
contains both \(ap\) and \(an\), so it contains \(ab\).

The case \(p<a<b\) follows from the case \(a<b<p\) by reversing the order
on \([n]\). The case \(P=\{1n,(n-1)n\}\) similarly follows from
\(P=\{12,1n\}\).
\end{proof}

There is one exceptional possibility: \(P\) may consist of two adjacent edges selected from the \(3\)-block. In this case \(G=\varnothing\) and
\(\operatorname{Asc}(G)=\{12,1n,(n-1)n\}\). Since we are not in Case 3, the third edge is not an atom of \([G,H]\), and hence \(M=\varnothing\).

We can now compute \(\Gamma(G,H)\) in all cases. 

\begin{lemma}\label{lem:join-below-top}
Let \(J\) be the join of the two atoms of \(P\). If \(J<H\), then
\([G,H]\) is contractible.
\end{lemma}

\begin{proof}
If the matching is empty, then \(\Gamma(G,H)\) is a 1-simplex. If the matching is nonempty, then, by Proposition~\ref{prop:forcing}, every atom of the matching \(M\) lies below \(J\). Hence the
join of all atoms is \(J\). If \(J<H\), then every subset of the atoms has join
strictly below \(H\). Thus the crosscut complex is a full simplex, and hence is
contractible.
\end{proof}

Suppose now that \(J=H\). We have two cases to consider, when the matching is not empty and when the matching is empty. 

\begin{lemma}\label{lem:cone-case}
If \(J=H\) and \(M\neq\varnothing\), then \(\Gamma(G,H)\) is a cone and hence \([G,H]\) is contractible.
\end{lemma}

\begin{proof}
A set of atoms has join \(H\) if and only if it contains both edges of \(P\).
Indeed, a set not containing both edges of \(P\) is a matching, so by Lemma~\ref{lem:matching-join} its join is its union with \(G\), which is strictly below \(H\).

Choose any \(e\in M\). If a set of atoms is a face of the crosscut complex, then
adding \(e\) does not make it contain both edges of \(P\). Thus \(e\) is a cone
point, so the crosscut complex is contractible.
\end{proof}

\begin{lemma}\label{lem:star-case}
If \(J=H\) and \(M=\varnothing\), then \(\Gamma(G,H)\) is
\(S^0\), and \([G,H]\) is facial.
\end{lemma}

\begin{proof}
There are exactly two atoms, and their join is \(H\). Hence the crosscut
complex consists of two isolated vertices, so it is \(S^0\).

Let \(p\) be the common vertex of the two edges in \(P\), and let
\(E_p=\{pq:q\neq p\}\). If \(P\) comes from the \(3\)-block, then
\(G=\varnothing\) and \(H=E_p\).

Otherwise \(P\) is a \(2\)-block of \(\operatorname{Asc}(G)\). Then \(p\) is
isolated in \(G\). Moreover, \(G\) has no other ascent: by Proposition~\ref{prop:forcing}, any such ascent would lie below \(H\) and would give a third atom of
\([G,H]\). Thus \(G\) is maximal among coconsistent graphs in which \(p\) is
isolated.

Put \(L=\{1,\ldots,p-1\}\) and \(R=\{p+1,\ldots,n\}\). A coconsistent graph with isolated \(p\) cannot contain both an edge inside \(L\) or \(R\) and an
edge between \(L\) and \(R\). Indeed, suppose that \(i<j<p\), \(ij\in G\),
and \(a<p<b\), \(ab\in G\). If \(a=i\) or \(a=j\), coconsistency fails
immediately at that vertex. If \(a<j\), coconsistency at \(i\) forces
\(ia\in G\), and then coconsistency fails at \(a\). If \(a>j\),
coconsistency at \(b\) forces \(ib\in G\), and then coconsistency fails at
\(i\). The case of an edge inside \(R\) is symmetric.

It follows that \(G\) consists either of all edges between \(L\) and \(R\), or
of all edges inside \(L\) and inside \(R\). In either case \(G\cup E_p\) is coconsistent. Every common upper bound of the two atoms contains \(E_p\), so \(H=J=G\cup E_p\). Thus \(H\setminus G=E_p\), and \([G,H]\) is
facial by Proposition~\ref{prop:facial-difference}.
\end{proof}

\subsubsection{Case 3: The 3-block}

In the last case \(A(G,H)=\{12,1n,(n-1)n\}\), we have \(G=\varnothing\) and coconsistency forces \(H=K_n\), so \([G,H]\) is the full facial interval.

Corollary~\ref{cor:matching-case} and Lemmas~\ref{lem:join-below-top}, \ref{lem:cone-case}, and~\ref{lem:star-case} exhaust the atom configurations of Proposition~\ref{prop:atom-types}. In every nonfacial case the crosscut complex is contractible. Since the sphericity of facial intervals is known, this proves Conjecture~\ref{conj:reiner} in corank 3.

\begin{theorem}\label{thm:main}
Let \(n\geq 4\). An interval in \(B(n,n-3)\) is spherical if and only if it is
facial. Every nonfacial proper interval is contractible.
\end{theorem}

\bibliographystyle{alpha-noeditors}
\bibliography{references}

\end{document}